\documentclass[a4,11pt]{article}
\usepackage{graphicx} 
\usepackage[left=1in,right=1in,top=1in,bottom=1.2in]{geometry}
\usepackage{amsfonts}
\usepackage{amsmath}
\usepackage{amssymb}
\usepackage{pgfplots}
\usetikzlibrary{intersections}
\usepackage{enumitem}
\usepackage{tikz}
\usepackage{amsthm}
\usepackage{mathrsfs}
\usepackage{csquotes}
\usepackage{booktabs}
\usepackage{amsmath}
\usepackage{graphicx}
\usepackage{chngcntr}
\usepackage{mathtools}
\usepackage{hyperref}
\usepackage[english]{babel}
\usepackage{comment}

\usepackage[capitalise]{cleveref}


\allowdisplaybreaks

\newcommand{\floor}[1]{\left\lfloor #1\right\rfloor}

\newcommand{\R}{\mathbb{R}}

\newcommand{\E}{\mathbb{E}}

\newcommand{\N}{\mathbb N}
\newcommand{\eps}{\varepsilon}
\counterwithin{equation}{section}
\newtheorem{theorem}{Theorem}

\newtheorem{lemma}[theorem]{Lemma}

\newtheorem{remark}{Remark}

\pgfplotsset{compat=1.17}

\usepackage{authblk}
\title{On the Fourier transform of random Bernoulli convolutions}
\author[1]{Simon Baker\thanks{Research supported in part by EPSRC grant number EP/W003880/1.}}
\author[2]{Henna Koivusalo}
\author[3]{Sascha Troscheit\thanks{Research supported in part by the European Research Council Marie
Sk\l{}odowska--Curie Personal Fellowship \#101064701.}}
\author[2]{Xintian Zhang}

\affil[1]{Department of Mathematical Sciences, Loughborough University, Loughborough, LE11 3TU,
United Kingdom}
\affil[2]{School of Mathematics, University of Bristol, Bristol, BS8 1UG, United Kingdom}
\affil[3]{Department of Mathematics, Uppsala University, Box 480, 751 06 Uppsala, Sweden}

\begin{document}

\maketitle

\begin{abstract}
  We investigate random Bernoulli convolutions, namely, probability measures given by the infinite convolution
  \[
    \mu_\omega = \mathop{\circledast}_{k=1}^{\infty} \left( \frac{\delta_0 + \delta_{\lambda_1 \lambda_2 \ldots \lambda_{k-1} \lambda_k}}{2} \right),
  \]
  where $\omega=(\lambda_k)$ is a sequence of i.i.d. random variables each following the uniform distribution on some fixed interval. We study the
  regularity of these measures and prove that when $\exp\E\left( \log \lambda_1\right)>\frac{2}{\pi},$ the Fourier transform $\widehat{\mu}_\omega$ is an $L^{1}$ function almost surely. This in
  turn implies that the corresponding random self-similar set supporting $\mu_{\omega}$ has
  non-empty interior almost surely. This improves upon a previous bound due to Peres, Simon and Solomyak. Furthermore, under no assumptions on the value of
  $\exp \E(\log \lambda_1),$ we prove that $\widehat \mu_\omega$ will decay to zero at a polynomial rate
  almost surely. 
\end{abstract}

\section{Introduction}

The distribution of the random series $\sum_{k}\pm\lambda^k$, where the signs are chosen independently with
equal probabilities, has been studied for almost 100 years. It was observed in 1935 by Jessen and
Wintner \cite{jessen1935distribution} that the resulting measure $\nu_\lambda$, now known as the
{\it Bernoulli convolution}, is always either absolutely continuous or singular with respect to the
Lebesgue measure. It is easy to see that for any $\lambda<1/2$, the support  of
$\nu_\lambda$ is a set of Lebesgue measure $0$, and hence $\nu_{\lambda}$ is automatically singular,
but for $\lambda>1/2$ the situation is much more subtle.

There has been much progress on Bernoulli convolutions over the past century.
Erd\H{o}s proved in \cite{Erdos1939} that whenever $\lambda\in (1/2,1)$ is the reciprocal of a Pisot number, then $\nu_{\lambda}$
is singular. Complementary to this, Soloymak proved in \cite{solomyak1995random} that $\nu_{\lambda}$ is absolutely
continuous for Lebesgue almost every $\lambda\in(1/2,1).$ This was subsequently improved upon by
Shmerkin in \cite{Shmerkin2019} who proved that the set of exceptions to this statement is not only
of Lebesgue measure zero, but in fact has zero Hausdorff dimension. Specific examples of algebraic
$\lambda$ for which $\nu_{\lambda}$ is known to be absolutely continuous are due to Garsia
\cite{Garsia}, Varju \cite{VarjuAC}, and Kittle \cite{Kittle}. Despite these advances, it is still
not known whether $\lambda\in (1/2,1)$ exists for which $\lambda^{-1}$ is not a Pisot number and $\nu_{\lambda}$ is singular. We mention for
completeness that much more is known about the dimension of Bernoulli convolutions. Building upon
work of Hochman \cite{Hochman14}, Varju \cite{Varju2019} proved that whenever $\lambda\in (1/2,1)$ is
transcendental then $\nu_{\lambda}$ has dimension $1$.

Given the above discussion, it is natural to wonder what can cause $\nu_\lambda$ to be singular. If
$\lambda$ has enough algebraic rigidity, such as in the case when $\lambda^{-1}$ is a Pisot number, the
sums $\sum\pm\lambda^k$ can begin to accumulate disproportionately on parts of the line, which causes the measure
$\nu_\lambda$ to be singular. However, this kind of algebraic rigidity is rare.
It is easy to remove this algebraic rigidity by choosing the parameters $\lambda$ \emph{randomly}.
To be more precise, we replace the measure $\nu_\lambda$ by the measure
\[
  \mu_{\omega} = \mathop{\circledast}_{k=1}^{\infty} \left( \frac{\delta_0 + \delta_{\lambda_1 \lambda_2 \ldots \lambda_{k-1} \lambda_k}}{2} \right),  
\]
where the terms in the sequence $\omega=(\lambda_k)$ are independent of each other and are each distributed according to the uniform
distribution on some closed interval $W\subset (0,1)$. The measures $\mu_{\omega}$ will be our main object of study. Intuitively, adding randomness in this fashion should rule out the algebraic rigidity that was observed in the deterministic case and so removes the cause for the irregularity of the distribution. However, the analysis in this random setting is different to the
deterministic case and comes with its own challenges.
For example, for the classical Bernoulli convolution $\nu_\lambda$, the support is an
interval whenever $\lambda\geq 1/2$ and the challenge is finding the exact distribution within the
interval.
However, for a random measure $\mu_{\omega},$ the geometry of the support
can be more complicated. One of our main objectives is to establish conditions on the random
model that guarantee the existence of interior points in the support of $\mu_{\omega}$. This problem of finding
interior points in parameterised families of random fractal sets, has been attracting significant attention
lately, see e.g. \cite{DekkingSimonSzekely2024,Feng2022TypicalSS,BaranyRams2025} and the references therein. This study of random measures is
also motivated by the open question of whether there exist self-similar sets in $\mathbb{R}$ with positive Lebesgue
measure but empty interior \cite{PerSol98}.

To properly formulate and contextualise our results, it is necessary to give some definitions and a review of existing results.
Let $W=[\lambda_{\min},
\lambda_{\max}]\subset (0,1)$ denote some closed interval. Let $\omega=(\lambda_{k})$ denote a
sequence whose entries are chosen from $W$ independently with respect to the uniform measure. To
each $\omega=(\lambda_{k})$ we associate the following random set $$\Lambda_{\omega}=\left\{\sum_{j=1}^{\infty}a_{j}\prod_{k=1}^{j}\lambda_{k}:a_{j}\in \{0,1\}\,\forall k\in \N\right\}.$$ It is easy to show that $\Lambda_{\omega}$ is the support of $\mu_{\omega}$ for any $\omega$. Moreover, if we let $\pi_{\omega}:\{0,1\}^{\N}\to \Lambda_{\omega}$ be the map given by $$\pi_{\omega}((a_j))=\sum_{j=1}^{\infty}a_{j}\prod_{k=1}^{j}\lambda_{k},$$ then $\mu_{\omega}=\pi_{\omega}\nu$ where $\nu$ is the $(\frac{1}{2},\frac{1}{2})$ Bernoulli measure on $\{0,1\}^{\N}$. $\Lambda_{\omega}$ can be interpreted as a random analogue of a self-similar set and the measure $\mu_{\omega}$ can similarly be interpreted as a random analogue of a self-similar measure. For more on random self-similar sets and random self-similar measures, we refer the reader to \cite{Koivusalo2013, Troscheit2017, BaranyKaenmakiRams2025}. Denote by
$$
\lambda_g=\exp\E\left( \log \lambda_1\right)
$$ 
the (geometric) expectation of the contraction rates. Note
that for any $\omega$ we have $\Lambda_\omega\subset [0,(1-\lambda_{\max})^{-1}]$. If $\lambda_{\max}<1/2$, then
$\Lambda_{\omega}$ is a Cantor set and $\mu_{\omega}$ is singular for all $\omega$. As in the
deterministic case, when $W\cap [1/2, 1]\neq \emptyset$ the question of absolute continuity and
other regularity properties of the measure become non-trivial. It is clear that they will depend on the position
and size of $W$ within $(0,1)$, and hence on the parameter $\lambda_{g}$. It is a consequence of the convergence part of the Borel-Cantelli lemma that when $\lambda_g<\tfrac
12$, the support $\Lambda_\omega$ has zero Lebesgue measure almost surely, and hence
$\mu_\omega$ is almost surely singular. It was shown by Peres, Simon and Solomyak in \cite{RandomBernoulli} that when
\(\lambda_g>\frac{1}{2}\) then \(\mu_\omega \ll \mathcal{L}\) with a density in
\(L^2(\mathbb{R})\) almost surely,  and furthermore, if \(\lambda_g>\frac{e^{1/2}}{2}\approx 0.824\) then
\(\mu_\omega \ll \mathcal{L}\) with a continuous density, almost surely. Consequently, if $\lambda_{g}>\frac{e^{1/2}}{2}$ then $\Lambda_{\omega}$ will almost surely have non-empty interior.

Even though this problem of finding interior is geometric in nature, routes to finding interior points often rely on
techniques from Fourier analysis. This is also the case in the work of Peres, Simon and Solomyak
\cite{RandomBernoulli}. Recall that the \emph{Fourier transform} of a Borel probability measure $\mu$ is defined
as 
\[
  \widehat\mu(\xi) = \int e^{-2\pi i \xi x}d\mu(x). 
\] 
The regularity of the Fourier transform of a fractal measure is an indicator of the `smoothness' of the
supporting fractal set, and is also an object of interest in its own right. Studying the Fourier analytic properties of the
deterministic Bernoulli convolution has proven to be a fundamental tool dating back to the early works of
Erd\H{o}s and Kahane \cite{Erdos1939, Kahane1979}. The study of Fourier transforms of deterministic
fractal measures continues to be an active topic. We refer the interested reader to \cite{Sahlsten} for a recent comprehensive
survey.

In the set-up of the random Bernoulli convolution, the techniques of Peres, Simon and
Solomyak \cite{RandomBernoulli} relied on a Sobolev dimension estimate for the measure $\mu_\omega$.
We are able to improve upon it by finding bounds for the $L^1$-norm of $\widehat \mu_\omega$
directly. Consequently, we are able to improve the threshold for interior points given in
\cite{RandomBernoulli} from $\frac{e^{1/2}}{2}\approx 0.824$ to $\frac{2}{\pi}\approx 0.636$. Our main result is the following.
\begin{theorem}\label{main:cor}
 If $\lambda_g>\tfrac 2\pi$ then \(\widehat\mu_\omega \in {L^1(\R)}\) almost
	surely. 
\end{theorem}

If \(\widehat{\mu} \in L^1(\mathbb{R})\), then
\(\mu\) is absolutely continuous with respect to the Lebesgue measure and has a continuous density. For a proof of this fact see \cite[Theorem 3.4]{mattila_2015}. Using this result we see that \cref{main:cor} immediately implies the following statement. 
\begin{theorem}\label{main}
  If $\lambda_g>\tfrac 2\pi$ then \(\mu_\omega \ll \mathcal{L}\) with a continuous density almost
  surely. Thus $\Lambda_{\omega}$ almost surely has non-empty interior.
\end{theorem}
We note that, in general, it is not the case that whenever some measure $\mu$ has continuous
density, it  must follow that $\widehat\mu\in L^1(\R)$. We also
emphasise that we cannot say whether the parameter $\frac{2}{\pi}$ appearing in 
\cref{main,main:cor} is optimal.
It is an interesting problem to determine what the optimal thresholds are for these theorems.

It turns out that our proof technique for \cref{main:cor} can be modified to give another result on
the regularity of $\mu_\omega$. We call a measure $\mu$ a \emph{Rajchman measure} if its Fourier
transform decays to $0$ as $|\xi|\to \infty$. We say that  the Fourier transform of a
measure $\widehat \mu$ has {\it polynomial decay}, if, for some $C, \rho>0$, 
\[
  |\widehat \mu(\xi)|\le C|\xi|^{-\rho}
\]
for all $\xi\neq 0$. Determining whether a measure is Rajchman, and if it Rajchman, the speed at which it converges to zero, is an important problem connecting many distinct areas of mathematics. For instance, it plays an important role in the uniqueness problem from Fourier analysis \cite{KL92,Salem}, detecting patterns in fractal sets \cite{LaPr}, and finding normal numbers in fractal sets \cite{DEL,PVZZ}. Understanding the decay properties of the Fourier transform of a deterministic fractal measure has received significant attention recently. We refer the reader to \cite{ACWW,ARHW,BakerSahlsten,BakBan25,BakerKhalilSahlsten,LPS25,LiSah20,LiSah22,Sahlsten,Streck} and the references therein for a sample of recent results. 
In this paper, we will prove that the Fourier transform of a random Bernoulli convolution will almost surely have polynomial decay. Moreover, the decay exponent can be chosen independently of $\omega$.
\begin{theorem}\label{main3}
There exists $\rho>0$ such that for almost all $\omega$, there exists $C>0$ such that 
$$|\hat{\mu}_{\omega}(\xi)|\le C|\xi|^{-\rho}$$
for all $\xi\neq 0$.
\end{theorem}

We finish this introductory section by remarking that it is possible to consider measures $\mu_{\omega}$ also in the case where $\lambda_{\max}>1$.
As long as $\lambda_g<1$, the arguments in the following sections go through with only minor
technical changes but for simplicity we restrict to the case where $\lambda_{\max}<1$.

What remains of the paper will be structured as follows. The proof of \cref{main:cor} is contained in \cref{sec:main}. We will then adapt the argument used in \cref{sec:main} to prove \cref{main3} in \cref{sec:main3}.

\section{Proof of  \cref{main:cor}}\label{sec:main}
In this section we show that for almost every $\omega \in W^{\mathbb{N}}$,  the measure  $\mu_{\omega}$ on
$\Lambda_\omega$ has an $L^1$ Fourier transform. Recall that elements of $W^{\mathbb N}$ are
sequences $(\lambda_k)$, and that the probability measure $\mathbb P$ on $W^{\mathbb{N}}$ is a
product of normalised Lebesgue measures each supported on the interval $W$. 

The following lemma expressing the Fourier transform as a product is standard and can be found in,
e.g.  \cite{mattila_2015}. We include the proof for the reader's convenience.
\begin{lemma}\label{lem:form} Let $\omega\in W^{\N}$. For every $\xi\in \R$ we have \[
  |\widehat{\mu}_\omega(\xi)|=\prod_{j=1}^{\infty}\left|\cos\left(\pi \xi\prod_{k=1}^{j}\lambda_{k}
\right)\right|.  \] \end{lemma}
\begin{proof}
 Let $\omega\in W^\N$ and $\nu$ be the $(\tfrac 12, \tfrac 12)$-Bernoulli measure on $\{0,1\}^{\N}.$ Recalling that $\mu_{\omega}=\pi_{\omega}\nu,$ we have
  \begin{align*}
    |\widehat{\mu}_\omega(\xi)|
    = \Bigg|\int e^{-2\pi i \xi x}d\mu_{\omega}(x)\Bigg|
    &=\Bigg|\int e^{-2\pi i \xi \sum_{j=1}^{\infty} a_{j}\prod_{k=1}^{j}\lambda_{k}}
    \,d{\nu((a_j))}\Bigg|\\
    &= \prod_{j=1}^{\infty}\left|\frac{1}{2}(1+e^{-2\pi i \xi \prod_{k=1}^{j}\lambda_{k} })\right|\\
    &=\prod_{j=1}^{\infty}\left|\frac{1}{2}(e^{\pi i \xi \prod_{k=1}^{j}\lambda_{k} }+e^{-\pi i
    \xi \prod_{k=1}^{j}\lambda_{k} })\right|\\
    &=\prod_{j=1}^{\infty}\left|\cos\left(\pi \xi\prod_{k=1}^{j}\lambda_{k} \right)\right|. 
  \end{align*} 
  Here the penultimate line follows from multiplying by $|e^{\pi i \xi \prod_{k=1}^{j}\lambda_{k} }|=1$. 
\end{proof}
It follows from \cref{lem:form} that to prove Theorem \ref{main:cor} we only need to show that 
\begin{equation}
  \int \prod_{j=1}^{\infty}\left|\cos\left(\pi \xi\prod_{k=1}^{j}\lambda_{k} \right)\right|\,d\xi<\infty
  \nonumber
\end{equation} for almost every $\omega$. We will estimate the average $L^1$-norm of $\widehat\mu_\omega$ from this expression more or less
directly, but we first need to establish some further notation.

It is clear from the above lemma that the $L^1$-bound for $\widehat \mu_\omega$ will likely depend upon the size of the products 
$\prod_{k=1}^{j}\lambda_{k}$. We will now quantify how these products behave in terms of $\lambda_g$. 

For $\eps>0, n\in \N$, define 
\[A_{\eps}{(n)}:=\left\{(\lambda_1,\ldots,\lambda_n)\in W^n \colon \lambda_g^{(1+\eps)n}\leq \prod_{k=1}^{n}
\lambda_k \leq \lambda_g^{(1-\eps)n }\right\}\] and 
\begin{equation}\label{eq:defB}
  B_{\eps}{(n)}:= \left\{(\lambda_{k})\in W^{\mathbb{N}}\colon\lambda_g^{(1+\eps)m}\leq \prod_{k=1}^{m}
  \lambda_k \leq \lambda_g^{(1-\eps)m }\, \text{ for all } m\geq n \right\}. 
\end{equation}
Notice that by the law of large numbers, for any $\eps>0$ we can make the $\mathbb P(B_\eps(n))$ as close to
$1$ as we like by taking $n$ large enough. 
\begin{lemma}\label{lem:reduce}
  Suppose that there exists $\eps>0$ such that 
  \[
    \int_{1}^{\infty}\int_{B_{\eps}(n)} \prod_{j=1}^{\infty}\left|\cos\left(\pi
    \xi\prod_{k=1}^{j}\lambda_{k} \right)\right|\,d \mathbb P\,d\xi<\infty
  \] 
  for all $n\in \mathbb{N}$. Then for almost every $\omega$ we have $\widehat{\mu}_{\omega}\in
  L^{1}(\R)$. 
\end{lemma}
\begin{proof}
  Notice first of all that if the assumption holds, then using the fact cosine is an even function,
  and since cosine is everywhere bounded from above by $1$, we have 
  \[
    \int_{\mathbb R}\int_{B_{\eps}(n)} \prod_{j=1}^{\infty}\left|\cos\left(\pi
    \xi\prod_{k=1}^{j}\lambda_{k} \right)\right|\,d \mathbb P\,d\xi<\infty. 
  \] 
  Now, by Fubini's theorem it follows that for any $n\in \mathbb{N},$ for almost every $\omega\in
  B_{\eps}(n)$ we have $\widehat{\mu}_{\omega}\in L^{1}(\R)$. By the law of large numbers, we have
  $\mathbb{P}(\cup_{n=1}^{\infty}B_{\eps}(n))=1$. Therefore  $\widehat{\mu}_{\omega}\in L^{1}(\R)$ for
  almost every $\omega$.
\end{proof}

For the time being, let $\eps>0$ be fixed. We will choose it at the end of our proof to guarantee
that the terms $(a_i)$ defined below are summable. We can now focus on the integral from 
\cref{lem:reduce}. We split the domain $[1,\infty)$ of $\xi$ into pieces using the powers
of $\lambda_g$: For $i\in
\mathbb{N}$, let
\[
  I_i:=[\lambda_g^{-i},\lambda_g^{-i-1}).
\] 
We now define, for any integer $i\in \N$ an integer $E_i$ given by the formula :
\[
  E_i=E(i,\eps):= \floor{(1-\eps)i},
\] 
Then, by \cref{lem:form}, for any $\xi\in \R$ we have
\[
  |\widehat{\mu}_\omega(\xi)|\leq \prod_{j=1}^{E_i}\left|\cos\left(\pi \xi\prod_{k=1}^{j}\lambda_{k} \right)\right|.
\]
For $n\in \mathbb N $ and $i$ sufficiently large that $E_i>n$, define 
$$C_{\eps}(n,E_i):=\left\{(\lambda_k)\in W^{E_i}\colon\lambda_{g}^{(1+\epsilon)j}\leq \prod_{k=1}^{j}\lambda_{k}\leq \lambda_{g}^{(1-\epsilon)j}\,\text{ for all }n\leq j \leq E_i\right\}$$
and $$D_{\epsilon}(n,E_i):=\left\{(\lambda_k)\in W^{\N}\colon(\lambda_1,\ldots,\lambda_{E_i})\in C_{\eps}(n,E_i)\right\}.$$

It is not hard to check that $B_\eps(n)\subseteq D_{\eps}(n,E_i)$. 
Let $\mathbb P_{E_i}$ denote the product measure on $W^{E_i}$ coming from the uniform distribution on $W$, or, equivalently, the projection of
$\mathbb{P}$ to the first $E_i$ coordinates. 

\begin{lemma}\label{lem:ai}
  For $\eps>0,$ $n\in \N$, and $i$ sufficiently large that $E_i>n$, denote 
  \[
    a_i=a_i(n, \eps):=\int_{I_i}\int_{C_{\eps}(n,E_i)}  \prod_{j=1}^{E_i}\left|\cos\left(\pi
    \xi\prod_{k=1}^{j}\lambda_{k} \right)\right|d \mathbb P_{E_i} d \xi.
  \] 
  The hypothesis of \cref{lem:reduce} is satisfied if there exists $\epsilon>0$ such that  $\sum_{i:E_i>n} a_i(n, \eps)<\infty$ for all $n\in \N$. 
\end{lemma}
\begin{proof}
  Notice that for all $\eps>0$ and $n\in \N$, 
  \begin{align*}
  	\int_{1}^{\infty}\int_{B_{\eps}(n)}  \prod_{j=1}^{\infty}\left|\cos\left(\pi
  	\xi\prod_{k=1}^{j}\lambda_{k} \right)\right|\,d \mathbb P\,d\xi
  	&= \sum_{i=1}^{\infty}
  	\int_{I_i}\int_{B_{\eps}(n)}  \prod_{j=1}^{\infty}\left|\cos\left(\pi
  	\xi\prod_{k=1}^{j}\lambda_{k} \right)\right|d \mathbb P d\xi\\
  	&=\sum_{i:E_i\leq n}
  	\int_{I_i}\int_{B_{\eps}(n)}  \prod_{j=1}^{\infty}\left|\cos\left(\pi
  	\xi\prod_{k=1}^{j}\lambda_{k} \right)\right|d \mathbb P d\xi\\
  	&+\sum_{i:E_i> n}
  	\int_{I_i}\int_{B_{\eps}(n)}  \prod_{j=1}^{\infty}\left|\cos\left(\pi
  	\xi\prod_{k=1}^{j}\lambda_{k} \right)\right|d \mathbb P d\xi.
  	\end{align*} The first summation in the above is a finite sum. Therefore to verify that the
	hypothesis of \cref{lem:reduce} is satisfied we only need to bound the second summation. This we do below:
  \begin{align*}
    \sum_{i:E_i> n}
    \int_{I_i}\int_{B_{\eps}(n)}  \prod_{j=1}^{\infty}\left|\cos\left(\pi
    \xi\prod_{k=1}^{j}\lambda_{k} \right)\right|d \mathbb P d\xi
    &\leq \sum_{i:E_i> n} \int_{I_i}\int_{D_{\eps}(n,E_i)}
    \prod_{j=1}^{\infty}\left|\cos\left(\pi \xi\prod_{k=1}^{j}\lambda_{k} \right)\right|d \mathbb P
    d\xi\\
    &\leq  \sum_{i:E_i> n} \int_{I_i}\int_{D_{\eps}(n,E_i)}  \prod_{j=1}^{E_i}\left|\cos\left(\pi
    \xi\prod_{k=1}^{j}\lambda_{k} \right)\right|d \mathbb P d\xi\\
    &= \sum_{i:E_i> n} \int_{I_i}\int_{C_{\eps}(n,E_i)}  \prod_{j=1}^{E_i}\left|\cos\left(\pi
    \xi\prod_{k=1}^{j}\lambda_{k} \right)\right|d \mathbb P_{E_i} d\xi.
  \end{align*}
  Thus, by our assumption $\sum_{i:E_i>n}a_{i}(n,\eps)<\infty$, the hypothesis of \cref{lem:reduce} is satisfied.
\end{proof}

By \cref{lem:ai}, to prove \cref{main:cor} we only need to find $\eps>0$ such that for all $n\in\N$ we have $\sum_{i:E_i>n} a_i(n, \eps)<\infty$. In order to find a threshold of $\lambda_g$ from which finding such an $\eps>0$
becomes possible, we need to understand the fine behaviour of the products $\prod_{k=1}^j \lambda_k$.
To that end, we define the following.

Let $\eps>0, i\in \N$ and $E_i=\lfloor (1-\eps)i\rfloor$ be fixed.  For a
given $M, h\in \mathbb{N}$, we write
\[
  V(\ell, M)=\Big[\arccos{\lambda_g^{\frac{h}{M}}},\arccos{\lambda_g^{\frac{h+1}{M}}}\Big)\cup
  \Big(\pi-\arccos{\lambda_g^{\frac{h+1}{M}}},\pi-\arccos{\lambda_g^{\frac{h}{M}}}\Big]. 
\]
We emphasise that we interpret $\arccos$ as a function from $[-1,1]\to[0,\pi]$. Furthermore, for $M,
k, h\in \N$ and $\xi \in \R$, we set
\[
  G_{h,k}=G_{h,k}(\xi, M):= \Big\{(\lambda_1,\ldots ,\lambda_k)\in W^{k}\colon\xi\prod_{j=1}^k
\lambda_j\pi\in V(h, M)+\mathbb{Z}\pi \Big\}. \]
Roughly speaking, the set $G_{h,k}$ consists of those elements of $W^{k}$ for which $|\cos(\pi
\xi \prod_{j=1}^k \lambda_j)|$ is approximately of value $\lambda_g^{h/M}$. Further, we set, for $(h_1, \ldots,
h_{E_i})\in \N^{E_i}$, 
$$ L_{h_1,\ldots,h_{E_i}}(\xi , M):=\left\{(\lambda_{1},\ldots, \lambda_{E_i})\in W^{E_i}\colon\xi\prod_{j=1}^k
\lambda_j\pi\in V(h_k, M)+\mathbb{Z}\pi\, \textrm{ for all }1\leq k\leq E_i\right\}.$$

For brevity's sake, we will denote $L_{h_1,\ldots,h_{E_i}}(\xi , M)$ by $L_{h_1,\ldots,h_{E_i}}$, however, we urge the reader to keep in mind that this set depends on $\xi$ in particular. 
 Recalling that the aim is to guarantee that $(a_i)$ gives a summable series, we look for an upper
bound for $a_i$. 

\begin{lemma}\label{lem:aibound}
  Let $\eps>0$, $M, n\in \N$ and let $i$ be sufficiently large that $n< E_i$. We have the following bound for $a_i$: 
  \begin{equation*}
    a_i(n, \eps)\leq \int_{I_i}\sum_{(h_1,\ldots,h_{E_i})\in\mathbb{N}^{E_i}}
    \lambda_g^{\sum_{j=1}^{E_i}\frac{h_j}{M}}\cdot \mathbb P_{E_i}(C_{\eps}(n,E_i)\cap L_{h_1,\ldots,h_{E_i}})
    d\xi. 
  \end{equation*}
  
\end{lemma}
\begin{proof}
  
  Notice that for any $(h_1, \ldots, h_{E_i})\in \mathbb{N}^{E_i}$, if $(\lambda_1,\ldots,\lambda_{E_i})\in
  L_{h_1,\ldots,h_{E_i}}\cap C_{\eps}(n, E_i) $ then for any $1\leq j \leq E_i$ we have
  \[
    \left|\cos\left(\pi \xi \prod_{k=1}^j \lambda_k\right)\right|\leq \lambda_g^{\frac{h_j}{M}}.
  \] 
  Hence, for any $(h_1, \ldots, h_{E_i})\in \mathbb{N}^{E_i}$ and $(\lambda_1,\ldots,\lambda_{E_i})\in
  L_{h_1,\ldots,h_{E_i}}\cap C_{\eps}(n, E_i) $, we have a uniform upper bound  
  \[
    \prod_{j=1}^{E_i}\Big|\cos\left(\pi \xi\prod_{k=1}^{j}\lambda_{k} \right)\Big|\leq
    \lambda_g^{\sum_{j=1}^{E_i}\frac{h_j}{M}}.
  \]
  It is easy to check that for $(h_1,\ldots,h_{E_i})\neq (h_1',\ldots,h_{E_i}')$, the sets
  $L_{h_1,\ldots,h_{E_i}}$ and $L_{h_1',\ldots,h_{E_i}'}$ are disjoint, and that they exhaust $W^\N$ up to a set of measure zero. Thus,
  for $E_i>n$, we have the following bound for $a_i(n,\epsilon)$ 
  \begin{align*}
    a_i(n,\epsilon)
    &= \int_{I_i}\sum_{(h_1,\ldots,h_{E_i})\in\mathbb{N}^{E_i}}\int_{C_{\eps}(n,E_i)\cap
      L_{h_1,\ldots,h_{E_i}}}  \prod_{j=1}^{E_i}\Bigg|\cos\left(\pi \xi \prod_{k=1}^{j}\lambda_{k}
    \right)\Bigg|d \mathbb P_{E_i} d\xi \\
    &\leq \int_{I_i}\sum_{(h_1,\ldots,h_{E_i})\in\mathbb{N}^{E_i}}
    \lambda_g^{\sum_{j=1}^{E_i}\frac{h_j}{M}}\cdot \mathbb P_{E_i}(C_{\eps}(n,E_i)\cap
    L_{h_1,\ldots,h_{E_i}}) d\xi.
    \qedhere
  \end{align*}
\end{proof}

  By \cref{lem:aibound} we would be in a good position to look for conditions on $\eps$ to make
  $(a_i(n, \eps))$ summable, if we could find bounds for the $\mathbb P_{E_i}(C_{\eps}(n,
  E_i)\cap L(h_1, \ldots, h_{E_i})) $ terms appearing in the integral. We do this via a conditioning argument, for which we need to define
  a good filtration of $\sigma$-algebras. 

We define three important events in the $\sigma$-algebra of $\mathbb P$.  The first
  arises from interpreting the sets of ``good'' contraction rates $G_{h_k,k}$ as events. Recall
  that we are considering $\xi\in \R$ and $M\in \N$ fixed, and that for $h, k\in \N$, $G_{h,
  k}=G_{h, k}(\xi, M)$. Now, write, for all $h, k\in \N$
  \[
    \mathcal{G}_{h,k}=\mathcal G_{h,k}(\xi, M) = 
    \left\{(\lambda_1,\lambda_2,\ldots) \in W^\N \colon (\lambda_1,\ldots,\lambda_{k})\in G_{h, k} \right \}. 
  \]
  Further, recall that the contractions in $A_\eps(k)$ are such that the $k$-fold product of
  contractions is $\eps$-close to the mean behaviour. We write $\mathcal{A}_\eps(k)$ for the analogous subset of $W^{\N}$:
  \begin{equation}\label{eq:defA}
    \mathcal{A}_ {\eps}(k) = \{(\lambda_1,\lambda_2,\ldots)\in W^\N \colon (\lambda_1,\ldots,\lambda_{k})\in
    A_{\eps}(k)\}.
  \end{equation}
  For $n\leq m$ we also define $\mathcal{F}_n(h_1,\ldots,h_m)$, to be the event of $\mathcal{G}_{h_k,k}$
  and $\mathcal{A}_\eps(k)$ both occurring from index $n$ to $m$ and $\mathcal{G}_{h_k,k}$
  occurring for all indices up to $n$, i.e.
  \begin{equation}\label{eq:calF}
    \mathcal{F}_n(h_1,\ldots,h_m) = \bigcap_{k=1}^{(n-1)} \mathcal{G}_{h_k,k}
    \;\cap\; \bigcap_{k=n}^m (\mathcal{G}_{h_k,k}\cap \mathcal{A}_{\eps}(k)).  
  \end{equation}
  Note specifically that $\mathcal{G}_{h_k,k}, \mathcal{A}_{\eps}(k),\mathcal{F}_n(h_1,\ldots,h_k) \in
  \mathfrak{F}_{k}$, where $\mathfrak{F}_k$ is the $\sigma$-algebra induced by the sequence of
  random contractions $\lambda_1,\lambda_2,\ldots,\lambda_k$. That is, $\mathfrak{F}_k =
  \sigma(\lambda_1,\ldots,\lambda_k)$ which is equal to the product of the projection of the full
  $\sigma$-algebra onto the first $k$ components with the trivial $\sigma$-algebra in the remaining
  components. 

  \begin{lemma}\label{lem:msr}
    Let $\eps>0$ and $n\in \N.$ For all $i$ sufficiently large that $E_i>n$ and $(h_1, \ldots, h_{E_i})\in \N^{E_i}$ 
    we have
    $$
      \mathbb P_{E_i}(L_{h_1,\ldots,h_{E_i}}\cap C_{\eps}(n,E_i))\leq
      \prod_{k=n}^{E_i}\mathbb P(\mathcal{G}_{h_k,k} \mid \mathcal{F}_{n}(h_1,\ldots,h_{k-1})).
    $$
  \end{lemma}

  \begin{proof}
    We can rewrite $\mathbb P_{E_i}(L_{h_1,\ldots,h_{E_i}}\cap C_{\eps}(n,E_i))$ in terms of the events
    above to get
    \begin{align*}
      \mathbb P_{E_i}(L_{h_1,\ldots,h_{E_i}}\cap C_{\eps}(n,E_i)) 
&= \mathbb P\left( \bigcap_{k=1}^{n-1} \mathcal{G}_{h_k,k}
\cap \bigcap_{k=n}^{E_i} (\mathcal{G}_{h_k,k}\cap \mathcal{A}_{\eps}(k)) \right)
= \mathbb P\left(  \mathcal{F}_n(h_1,\ldots,h_{E_i}) \right)
    \end{align*}
    We will use that $(\mathfrak{F}_n)_{n\geq 1}$ is a filtration and use the tower property of conditional
    expectations (and hence conditional probabilities), to write \begin{align*}
      \mathbb P(\mathcal{F}_n(h_1,\ldots,h_m)) 
  &= \mathbb P((\mathcal{G}_{h_m,m}\cap
  \mathcal{A}_\eps(m))\cap\mathcal{F}_{n}(h_1,\ldots,h_{m-1}))
  \\
  &= \mathbb P(\mathcal{G}_{h_m,m}\cap
  \mathcal{A}_\eps(m) \mid \mathcal{F}_{n}(h_1,\ldots,h_{m-1}))
  \cdot \mathbb P(\mathcal{F}_{n}(h_1,\ldots,h_{m-1}))
    \end{align*}
    for $m\geq n$.
    Repeatedly applying the identity above yields the following:
    \begin{align*}
      \mathbb P_{E_i}(L_{h_1,\ldots,h_{E_i}}\cap C_{\eps}(n,E_i)) 
  &= \mathbb P\left(  \mathcal{F}_n(h_1,\ldots,h_{E_i}) \right)
  \\
  &= \left(\prod_{k=n+1}^{E_i}\mathbb P(\mathcal{G}_{h_k,k}\cap
    \mathcal{A}_\eps(k) \mid \mathcal{F}_{n}(h_1,\ldots,h_{k-1}))
  \right)\cdot \mathbb P(\mathcal{F}_{n}(h_1,\ldots,h_{n}))
  \\
   &\leq
   \prod_{k=n+1}^{E_i}\mathbb P(\mathcal{G}_{h_k,k}\cap
   \mathcal{A}_\eps(k) \mid \mathcal{F}_{n}(h_1,\ldots,h_{k-1}))
   \\
   &\leq
   \prod_{k=n+1}^{E_i}\mathbb P(\mathcal{G}_{h_k,k} \mid \mathcal{F}_{n}(h_1,\ldots,h_{k-1})).
\end{align*}
This completes our proof.
\end{proof}

The following lemma provides an upper bound for the probabilities appearing in the product in \cref{lem:msr}. 
\begin{lemma}\label{lem:value}
  Let $0<\eps<1$, $n\in \N$, $i$ be sufficiently large that $E_i>n$ and $n< k\leq E_i$.
  Denote $\Delta=\lambda_{\max}-\lambda_{\min}$. Then for \(\xi\in I_i\) and $(h_1, \ldots, h_{k})\in \N^{k}$ we have
  \[
    \mathbb{P}( \mathcal{G}_{h_k,k}(\xi, M)\mid \mathcal{F}_{n}(h_1,\ldots,h_{k-1}))\leq
    \left(1+\frac{3\lambda_g^{\eps^2 i}}{\Delta} \right) \cdot
    \frac{\arccos{\lambda_g^{\frac{h_k+1}{M}}}-\arccos{\lambda_g^{\frac{h_k}{M}}} }{\pi/2}.
  \]
\end{lemma}

\begin{proof}
  Let $\eps, n,i$ and $k$ be as in the statement of our lemma. Let $\xi\in I_i$ and $(h_1,\ldots, h_k)\in \N^{k}$. For $(\lambda'_1,\ldots,\lambda'_{k-1},\ldots)\in\mathcal{F}_{n}(h_1,\ldots,h_{k-1})$
  define 
  \begin{align*}
    S(\lambda_1', \ldots, \lambda_{k-1}'):=&\{\lambda_k\in W\colon(\lambda_1,\ldots,\lambda_k)\in
    G_{h_k,k} \text{ and } (\lambda_1,\ldots,\lambda_{k-1})=(\lambda'_1,\ldots,\lambda'_{k-1}) \}. 
  \end{align*}
  The set $S(\lambda_1',\ldots,\lambda_{k-1}')$ can be expressed explicitly as follows 
  \begin{align*}
    S&(\lambda_1', \ldots, \lambda_{k-1}')
    =\Big\{\lambda_k\in W \colon \pi\xi\lambda_k \prod_{j=1}^{k-1}\lambda_j'\in
    V(k,M)+\mathbb{Z}\pi\Big\}\\
     &=W\cap \bigcup_{l\in\mathbb{Z}}
     \Bigg(\bigg[\frac{\arccos{\lambda_g^{\frac{h_k}{M}}}}{\pi\xi\prod_{j=1}^{k-1}\lambda_j'},
       \frac{\arccos{\lambda_g^{\frac{h_k+1}{M}}}}{\pi\xi\prod_{j=1}^{k-1}\lambda_j'}\bigg)+
     \frac{l}{\xi\prod_{j=1}^{k-1}\lambda_j'}\\
     &\hspace{5cm}
       \cup
       \bigg(\frac{\pi-\arccos{\lambda_g^{\frac{h_k+1}{M}}}}{\pi\xi\prod_{j=1}^{k-1}\lambda_j'},
       \frac{\pi-\arccos{\lambda_g^{\frac{h_k}{M}}}}{\pi\xi\prod_{j=1}^{k-1}\lambda_j'}\bigg] +
     \frac{l}{\xi\prod_{j=1}^{k-1}\lambda_j'}\Bigg)\\
     &=:W\cap \bigcup_{l\in\mathbb Z} \big(V_+(l,k)\cup V_-(l,k) \big), 
  \end{align*}
  where, for $l\in \mathbb Z$,
  \[
    V_+(l,k)=\bigg[\frac{\arccos{\lambda_g^{\frac{h_k}{M}}}}{\pi\xi\prod_{j=1}^{k-1}\lambda_j'},
    \frac{\arccos{\lambda_g^{\frac{h_k+1}{M}}}}{\pi\xi\prod_{j=1}^{k-1}\lambda_j'}\bigg)
    +\frac{l}{{\xi\prod_{j=1}^{k-1}\lambda_j'}}
  \]
  and
  \[
    V_-(l,k)=\bigg(\frac{\pi-\arccos{\lambda_g^{\frac{h_k+1}{M}}}}{\pi\xi\prod_{j=1}^{k-1}\lambda_j'},
    \frac{\pi-\arccos{\lambda_g^{\frac{h_k}{M}}}}{\pi\xi\prod_{j=1}^{k-1}\lambda_j'}\bigg]
    +\frac{l}{{\xi\prod_{j=1}^{k-1}\lambda_j'}}. 
  \]
  From the above expression it is clear that \(S(\lambda_1', \ldots, \lambda_{k-1}')\) is a union of
  some large number $P$ of intervals of the form $V_-(l,k)$ and $V_+(l,k)$  which have the same
  length
  \begin{equation}\label{eq:J00}
    J= \frac{
    \arccos{\lambda_g^{\frac{h_k+1}{M}}}-\arccos{\lambda_g^{\frac{h_k}{M}}}}{\pi\xi\prod_{j=1}^{k-1}\lambda_j'},
  \end{equation}
  and at most $2$
  intervals of length less than $J$ (if the endpoints of $W$ ``cut off" a smaller piece of some
  $V_-(l,k)$ or $V_+(l,k)$). As
  $(\lambda'_1,\ldots,\lambda'_{k-1},\ldots)\in\mathcal{F}_{n}(h_1,\ldots,h_{k-1})$, we have
  $\prod_{\ell=1}^{k-1} \lambda_\ell' \geq \lambda_g ^{(1+\eps)(k-1)}.$ This implies the following
  estimate for $J$: 
  \begin{equation}\label{eq:J0}
    J= \frac{ \arccos{\lambda_g^{\frac{h_k+1}{M}}}-\arccos{\lambda_g^{\frac{h_k}{M}}}}{\pi
    \xi \prod_{j=1}^{k-1}\lambda_j'}\leq \frac{
  \arccos{\lambda_g^{\frac{h_k+1}{M}}}-\arccos{\lambda_g^{\frac{h_k}{M}}}}{\pi\xi
\lambda_g^{(1+\eps)(k-1)}}. 
  \end{equation}
  Since \(\xi\in I_i=[\lambda_g^{-i},\lambda_g^{-i-1})\), we have \(\xi\geq \lambda_g^{-i}\). Thus,
  using the fact that \(k-1<E_i\) and \cref{eq:J0}, we also have the following bound on $J$:
  \begin{equation}\label{eq:J1}
    J\leq \frac{ \arccos{\lambda_g^{\frac{h_k+1}{M}}}-\arccos{\lambda_g^{\frac{h_k}{M}}}}{\pi\xi
    \lambda_g^{(1+\eps)(k-1)}}\leq \frac{
  \arccos{\lambda_g^{\frac{h_k+1}{M}}}-\arccos{\lambda_g^{\frac{h_k}{M}}}}{\pi
    \lambda_g^{-i} \lambda_g^{(1+\eps)E_i}}.
  \end{equation}
  Recall now that \(E_i=\floor{(1-\eps)i}\leq (1-\eps)i\). Combining this bound with \cref{eq:J1} we obtain
  \begin{equation}\label{eq:J}
    J\leq \frac{ \arccos{\lambda_g^{\frac{h_k+1}{M}}}-\arccos{\lambda_g^{\frac{h_k}{M}}}}{\pi
    \lambda_g^{-i} \lambda_g^{(1+\eps)E_i}} \leq \frac{
  \arccos{\lambda_g^{\frac{h_k+1}{M}}}-\arccos{\lambda_g^{\frac{h_k}{M}}}}{\pi
\lambda_g^{-{\eps}^2i}}.
\end{equation}
We now begin estimating the number $P$ of intervals in $S(\lambda_1', \ldots, \lambda'_{k-1})$. Denote 
\[
  l_0= \min \{l\in\mathbb{Z} \colon V_+(l,k)\subseteq W \},
  \quad\text{and}\quad l_1= \max \{l\in\mathbb{Z} \colon V_-(l,k)\subseteq W \}. 
\]
We have
\begin{equation}\label{eq:P0}
   P\leq 2(l_1-l_0+1)+4.
\end{equation}
We need to find the maximal range of $l\in \mathbb Z$ for which $V_+(l,k)$ and $V_-(l,k)$ can fit inside $W$. Recalling the definition of these sets, if we divide $\Delta$, the length
of \(W\), by \((\xi\prod_{j=1}^{k-1}\lambda_j')^{-1}\), the distance between consecutive intervals in this collection, this yields the following upper bound for \(l_1-l_0+1\):
\begin{equation}\label{eq:P1}
   l_1-l_0+1\leq
  \left\lfloor\frac{\Delta}{(\xi\prod_{j=1}^{k-1}\lambda_j')^{-1}}\right\rfloor.
\end{equation}
Thus, by \cref{eq:P0} and \cref{eq:P1} we have 
\begin{equation}\label{eq:P}
  P\leq
  2\floor{\Delta\xi\prod_{j=1}^{k-1}\lambda_j'}+4.
\end{equation}
Substituting the value for $J$ given by \cref{eq:J00} into equation \cref{eq:P}, we obtain
\begin{equation}\label{eq:Pf}
  P \leq 2\floor{\frac{\Delta\cdot
  (\arccos{\lambda_g^{\frac{h_k+1}{M}}}-\arccos{\lambda_g^{\frac{h_k}{M}}} )}{J\pi}}+4\leq
  \frac{2\Delta\cdot (\arccos{\lambda_g^{\frac{h_k+1}{M}}}-\arccos{\lambda_g^{\frac{h_k}{M}}}
  )}{J\pi}+4.
\end{equation}

Given some $(\lambda_{1}',\ldots,\lambda_{k-1}',\ldots)\in \mathcal{F}_{n}(h_1,\ldots,h_{k-1})$
the contraction ratio $\lambda_k$ is freely chosen from $W=[\lambda_{\min},\lambda_{\max}]$. We can
now directly bound from above the Lebesgue measure of the set $S(\lambda_1', \ldots, \lambda_{k-1}')$
and normalize by the measure of $W$ (to obtain the distribution of the random variable $\lambda_k$).
From \cref{eq:J} and \cref{eq:Pf}, we obtain
\begin{align*}
  \frac{\mathcal{L}(S(\lambda_1',\ldots,\lambda_{k-1}'))}{\mathcal L(W)} 
  &\leq \frac{J\times (P+2)}{\Delta}\\
  &\leq \frac{J\cdot \left(\frac{2\Delta\cdot (\arccos{\lambda_g^{\frac{h_k+1}{M}}}-\arccos{\lambda_g^{\frac{h_k}{M}}} )}{J\pi}+6\right) }{\Delta}\\
  &\leq \frac{\arccos{\lambda_g^{\frac{h_k+1}{M}}}-\arccos{\lambda_g^{\frac{h_k}{M}}} }{\pi/2} + \frac{6J}{\Delta}\\
  &\leq \frac{\arccos{\lambda_g^{\frac{h_k+1}{M}}}-\arccos{\lambda_g^{\frac{h_k}{M}}} }{\pi/2}+ \frac{3\lambda_g^{\eps^2 i}}{\Delta}\cdot \frac{\arccos{\lambda_g^{\frac{h_k+1}{M}}}-\arccos{\lambda_g^{\frac{h_k}{M}}} }{\pi/2}\\
  &=\left(1+\frac{3\lambda_g^{\eps^2 i}}{\Delta} \right) \cdot \frac{\arccos{\lambda_g^{\frac{h_k+1}{M}}}-\arccos{\lambda_g^{\frac{h_k}{M}}} }{\pi/2}.
\end{align*}
This estimate holds uniformly for all choices of \((\lambda_1',\ldots,\lambda_{k-1}',\ldots)\in \in \mathcal{F}_{n}(h_1,\ldots,h_{k-1})\). We now recall that the probability measure $\mathbb P$ is a product of normalised Lebesgue measures on the infinite product $W^\N$. Thus,  
\begin{align}
  \mathbb{P}(\mathcal{G}_{h_k,k}\mid\mathcal{F}_{n}(h_1,\ldots,h_{k-1}))
    &\leq \left(1+\frac{3\lambda_g^{\eps^2 i}}{\Delta} \right) \cdot
    \frac{\arccos{\lambda_g^{\frac{h_k+1}{M}}}-\arccos{\lambda_g^{\frac{h_k}{M}}} }{\pi/2}. 
    \nonumber
    \qedhere
\end{align}
\end{proof}

With these probability estimates, we can establish good bounds for $(a_i(n,\epsilon))$. We now return to the
question of summability of $(a_i(n,\epsilon))$, which by \cref{lem:reduce,lem:ai} implies \cref{main:cor}.

\begin{proof}[Proof of  \cref{main:cor}]

 We the estimates from \cref{lem:aibound,lem:msr,lem:value}. Since these bounds
  do not depend upon the choice of \(\xi\in I_i\), for \(\eps>0, n\in\mathbb{N}\) and \(i\)
  large enough such that \(E_i>n\) we have
\begin{align}
    a_i\le& \int_{I_i}\sum_{(h_1,\ldots,h_{E_i})\in\mathbb{N}^{E_i}} \lambda_g^{\sum_{j=1}^{E_i}\frac{h_j}{M}}\cdot \mathbb P_{E_i}(C_{\eps}(n,E_i)\cap L(h_1,\ldots,h_{E_i})) d\xi\nonumber\\ 
    \leq& \int_{I_i}\sum_{(h_1,\ldots,h_{E_i})\in\mathbb{N}^{E_i}} \lambda_g^{\sum_{j=1}^{E_i}\frac{h_j}{M}}\cdot\prod_{j=n+1}^{E_i}\left(1+\frac{3\lambda_g^{\eps^2 j}}{\Delta} \right)\cdot \left(\frac{
       \arccos{\lambda_g^{\frac{h_j+1}{M}}}-\arccos{\lambda_g^{\frac{h_j}{M}}}
   }{\pi/2}\right)d\xi\nonumber\\
\leq&\left(\left(\frac{1}{\lambda_g}\right)^{i+1}-\left(\frac{1}{\lambda_g}\right)^{i}\right)
\nonumber\\
    &\hspace{2cm}\cdot\sum_{(h_1,\ldots,h_{E_i})\in\mathbb{N}^{E_i}} \lambda_g^{\sum_{j=1}^{E_i}\frac{h_j}{M}}\cdot\prod_{j=n+1}^{E_i}\left(1+\frac{3\lambda_g^{\eps^2 j}}{\Delta} \right)\cdot \left(\frac{
       \arccos{\lambda_g^{\frac{h_j+1}{M}}}-\arccos{\lambda_g^{\frac{h_j}{M}}}
   }{\pi/2}\right)\nonumber\\
   =&\left(\left(\frac{1}{\lambda_g}\right)^{i+1}-\left(\frac{1}{\lambda_g}\right)^{i}\right)  \cdot
   \prod_{i=1}^{\infty}\left(1+\frac{3\lambda_g^{\eps^2 j}}{\Delta} \right)\nonumber\\
    &\hspace{1em}\cdot \sum_{(h_1,\ldots,h_{n})\in\mathbb{N}^{n}} \lambda_g^{\sum_{j=1}^{n}\frac{h_j}{M}}
    \sum_{(h_{n+1},\ldots,h_{E_i})\in\mathbb{N}^{E_i-n}} \lambda_g^{\sum_{j=n+1}^{E_i}\frac{h_j}{M}}\cdot \prod_{j=n+1}^{E_i}\left(\frac{
       \arccos{\lambda_g^{\frac{h_j+1}{M}}}-\arccos{\lambda_g^{\frac{h_j}{M}}}
   }{\pi/2}\right)\nonumber\\
    \leq&  C_n\cdot \left(\frac{\sum_{k=0}^{\infty} \lambda_g^{\frac{k}{M}}\cdot
       \left(\arccos{\lambda_g^{\frac{k+1}{M}}}-\arccos{\lambda_g^{\frac{k}{M}}}\right)
   }{\lambda_g^{\frac{i}{E_i-n}}\pi/2}\right)^{E_i-n} \label{eq:sumai}
  \end{align}
  where $$C_n=\left(\frac{1}{\lambda_g}-1\right)\sum_{(h_1,\ldots,h_{n})\in\mathbb{N}^{n}} \lambda_g^{\sum_{j=1}^{n}\frac{h_j}{M}} \times  \prod_{i=1}^{\infty}\left(1+\frac{3\lambda_g^{\eps^2 j}}{\Delta} \right)<\infty.$$
  Now we focus on the final line of \cref{eq:sumai}. For convenience, we denote
  \(f(x):=\arccos{\lambda_g^{x}}\). Notice that  the first derivative of $f$ is:
  \[
    f'(x) = \frac{d}{dx} \arccos(\lambda_g^x) = -\frac{\lambda_g^x \ln \lambda_g}{\sqrt{1 -
    \lambda_g^{2x}}}>0,
  \]
  and it is a straightforward computation to check that the second derivative of  $f''$ is
  negative. Hence, $f$ is an increasing concave function, which together with the Mean Value
  Theorem implies that for any \(k\in\mathbb{N}\)
  \[
    f'(\tfrac{k+1}{M})\le \frac{f(\frac {k+1}M)-f(\frac kM) }{\tfrac 1M}\le f'(\tfrac kM) .
  \]
  Using this approximation, and the definition of integration for $x\mapsto
  \lambda_g^xf'(x)$, it immediately follows that the series becomes an integral
  \begin{align*}
    \lim_{M\to\infty}\sum_{k=0}^{\infty} \lambda_g^{\frac{k}{M}}\cdot
    \left(\arccos{\lambda_g^{\frac{k+1}{M}}}-\arccos{\lambda_g^{\frac{k}{M}}}\right)
	    &= \lim_{M\to\infty}\sum_{k=0}^{\infty} \frac 1M\cdot\lambda_g^{\frac{k}{M}}\cdot
	    f'(\tfrac kM)\\
	    &=\int_{0}^{\infty} -\frac{\lambda_g^{2x} \ln{\lambda_g}}{\sqrt{1 - \lambda_g^{2x}}} dx.
  \end{align*}
  Using the substitution \( u = \lambda_g^x \), we obtain 
  \[
    \int_{0}^{\infty} -\frac{\lambda_g^{2x} \ln{\lambda_g}}{\sqrt{1 - \lambda_g^{2x}}}
    dx=\int_{0}^{1} \frac{u}{\sqrt{1 - u^2}}  du= 1.
  \]
  Summarising, we have shown that 
  \[\lim_{M\to\infty}\sum_{k=0}^{\infty} \lambda_g^{\frac{k}{M}}\cdot
  \left(\arccos{\lambda_g^{\frac{k+1}{M}}}-\arccos{\lambda_g^{\frac{k}{M}}}\right)= 1.\]
  Recall that $E_i=\lfloor (1-\eps)i\rfloor$. Notice that since $\lambda_g>\frac{2}{\pi}$, there
  exists \(\eps'>0\) such that for all \(\eps<\eps'\) and all sufficiently large i we have
  \(\lambda_g^{\frac{i}{E_i-n}}>\frac{2}{\pi}\). Thus, for \(\lambda_g>\frac{2}{\pi}\), we can find
  \(M\)  and $\epsilon$ so that for $i$ large enough we have \[
    \frac{\sum_{k=0}^{\infty} \lambda_g^{\frac{k}{M}}\cdot
      \left(\arccos{\lambda_g^{\frac{k+1}{M}}}-\arccos{\lambda_g^{\frac{k}{M}}}\right)
    }{\lambda_g^{\frac{i}{E_i-n}}\pi/2}<1.
  \]
 In conclusion, by \cref{eq:sumai}, for $\lambda_g>\frac{2}{\pi}$, we can find $\eps$ and $M$
  such that for all \(n\in\mathbb{N}\) and \(i\) large enough depending on \(n\) we have \(a_i(n,\eps)\leq
  C_n\gamma^{E_i}\) where \(\gamma=\gamma(\eps, M, n)<1\). Since \(E_i=\floor{(1-\eps)i}\) and
  \((\gamma^{E_i})_{E_i}\) is summable, it follows that $(a_i(n,\eps))_i$ is summable. Thus, by \cref{lem:reduce} and
  \cref{lem:ai}, we complete the proof of \cref{main:cor}.  
\end{proof}

\begin{remark}
  We note that a method similar to the above could be applied on contraction ratios
  $(\lambda_k)$ distributed according to another probability which is absolutely continuous
  with respect to Lebesgue measure. The threshold for achieving $\widehat \mu_\omega\in L^1$
  then changes accordingly. 
\end{remark}

\section{Proof of \cref{main3}}\label{sec:main3}
The argument presented in this section can be viewed as an adaptation of the classical
Erd\H{o}s-Kahane argument \cite{Kahane1979, Erdos1939} to the random setting.

We are aiming to show that for almost all $\omega\in W^{\mathbb{N}}$, there exist some $\rho,
C>0$, such that for all $\xi\neq 0$ we have
\[|\widehat{\mu}_\omega(\xi)|\leq C|\xi|^{-\rho}.\]
Recall the notation from \cref{sec:main}. In particular, given $\eps>0$, for each $i\in \N$ we let let $E_i=E_i(\epsilon)=\lfloor (1-\eps)i\rfloor$. 

Set \(X:\mathbb{N} \to \mathbb{N}\) to be a function to be determined later. For any \(i\in \mathbb{N}\) and \(l\in\{0,\ldots,X(i)-1\}\), we let 
\[
  \xi_{i,l}:=\frac{1}{\lambda_g^i}+\frac{1}{\lambda_g^i}\left(\frac{1}{\lambda_g}-1\right)\cdot
  \frac{l}{X(i)}.
\]
The points $\xi_{i, 0}, \ldots, \xi_{i, X(i)-1}\in I_{i}$ are $\eta_{i}$-dense in $I_{i}$ where $\eta_i$ is given by 
\[
  \eta_i:= \frac{\lambda_g^{-i-1}-\lambda_g^{-i}}{X(i)}.
\]
Given $i\in \mathbb{N}$ we let
\[
  S_i:=\{\xi_{i,l}\colon l\in\{0,\ldots,X(i)-1\} \}.
  \]
 
For any $\omega\in W^{\N}$ the function $\hat{\mu}_{\omega}$ is Lipschitz continuous. It is a consequence of this property that to prove \cref{main3}, it is sufficient to establish the desired polynomial decay on a suitably dense countable subset of frequencies. The significance of this reduction is that it allows us to meaningfully apply the Borel-Cantelli lemma. The sequence $(\xi_{i,l})_{i,l}$ will take on the role of this countable subset. We emphasise that the density of this sequence is determined by the function $X$. With this strategy in mind, we begin by studying the behaviour of $\hat{\mu}_{\omega}$ only along the fixed sequence of
frequencies $(\xi_{i,l})_{i,l}$. As we have seen in the previous section, an upper bound for $|\widehat{\mu}_{\omega}(\xi)|$ can be derived from knowledge on the
behaviour of the products $\pi\xi\prod_{k=1}^j \lambda_k$ modulo one. In particular, \cref{lem:form} shows that if these products often take values away from $\mathbb{Z}\pi,$ then this gives a strong upper bound for $|\hat{\mu}_{\omega}(\xi)|$. With this observation in mind, for  $\theta\in (0,\pi/2)$ let
\[ V_{\theta}:=[\theta,\pi-\theta]. 
\] The next lemma formalises this connection. 

\begin{lemma}\label{PD*}
 Let $\alpha\in(0,1)$ and $\theta\in (0,\frac{\pi}{2})$. If $\omega=(\lambda_k)\in W^\mathbb{N}$
  satisfies
  \begin{equation}\label{eq:suff}
    \liminf_{i\to \infty} \min_{\xi_{i,l}\in S_i}\frac{\#\{j\leq
    E_i: \pi\prod_{k=1}^j \lambda_k \xi_{i,l} \in V_{\theta}+ \mathbb{Z}\pi   \} }{i}\geq \alpha  , 
  \end{equation}
  then for some $\rho>0$ depending upon $\alpha, \theta$ and $\lambda_{g}$,  if $i$ is large enough depending on $\omega,$ then for all $\xi_{i,l}\in
  S_i$ we have $|\widehat{\mu}_{\omega}(\xi_{i,l})|\leq \xi_{i,l}^{-\rho}$. 
\end{lemma}
\begin{proof}
  Let $\omega=(\lambda_k)\in W^{\mathbb{N}}$ be such that there exist a pair of parameters $\alpha$
  and $\theta$ such that
  \begin{equation*}  
    \liminf_{i\to \infty}\min_{\xi_{i,l}\in S_i}\frac{\#\{j\leq
    E_i: \pi\prod_{k=1}^j \lambda_k \xi_{i,l} \in V_{\theta}+ \mathbb{Z}\pi   \} }{i}\geq \alpha. 
  \end{equation*}
  Then by \cref{lem:form},  for all large enough $i$ and all \(l\in\{0,\ldots,X(i)-1\} \), we have
  \begin{equation}\label{eq:polbound}
    |\widehat{\mu}_\omega(\xi_{i,l})|\leq \prod_{j=1}^{E_i} \left|\cos\left(\pi \prod_{k=1}^j
    \lambda_k \xi_{i,l}\right)\right|\le (\cos \theta)^{i\alpha}. 
  \end{equation}
 Recalling the definition of $I_i$ and that $S_i\subset I_i$ for all $i$, we know that 
  \[
   \lambda_g^{-i} \leq \xi_{i,l}\leq \lambda_g^{-(i+1)}. 
  \]
  Therefore, for $i$ large enough, since $\xi_{i,l}\in S_i$ we can rewrite \cref{eq:polbound} in
  terms of $\xi_{i,l}$ as
  \[
    |\widehat{\mu}_\omega(\xi_{i,l})|\leq {(\lambda_g^{-i})}^{-\frac{\alpha\log{\cos\theta}}{\log
    \lambda_g}}= {(\lambda_g^{-i-1})}^{-\frac{i\alpha\log{\cos\theta}}{(i+1)\log \lambda_g}}\leq
    \xi_{i,l}^{-\frac{i\alpha\log{\cos\theta}}{(i+1)\log \lambda_g}}\leq
    \xi_{i,l}^{-\frac{\alpha\log{\cos\theta}}{2\log \lambda_g}}.
    \qedhere
  \]
\end{proof}
In summary, \cref{PD*} demonstrates that for all $i$ sufficiently large, if the
proportion of $j$ for which \(\pi\prod_{k=1}^j \lambda_k \xi_{i,l}\) belongs to a fixed
region bounded away from $\mathbb Z \pi$ exceeds \(\alpha\) for any $\xi_{i,l}\in S_i$, then $\widehat{\mu}_{\omega}$ will eventually satisfy a polynomial decay rate along the sequence \((\xi_{i,l})_{i,l}\).
Because of this, the probability of the event \(\pi\prod_{k=1}^{j}\lambda_k \xi_{i,l}\in
V_{\theta}+\mathbb{Z}\pi\) is important. In \cref{sec:main}, we studied a similar question of
`hitting probability' in \cref{lem:value}. The argument that we will now give follows a similar outline.

Recall the definitions of \(B_{\eps}(n)\) from \cref{eq:defB}  and $\mathcal
A_\eps(k)$ from \cref{eq:defA}. Let \(\theta\in (0,\frac{\pi}{4})\) be given, and let
\(j,i\in\mathbb{N}\), \(l\in \{0,\ldots,X(i)-1\}\) and \(a\in\{0,1\}\).
Define $\mathcal{I}_{j}(\xi_{i,l},\theta)$ and $\Xi_{j,i,l}(a)$ by setting 
\[
  \mathcal{I}_{j}(\xi_{i,l},\theta)=\left\{(\lambda_1, \lambda_2,\ldots)\in W^{\N}\colon \pi\prod_{k=1}^j \lambda_k
  \xi_{i,l} \in V_{\theta}+ \mathbb Z\pi\right\}.   	
\]  
and
\[
  \Xi_{j,i,l}(a)=\begin{cases}
    \mathcal{I}_{j}(\xi_{i,l},\theta) ,& \text{if } a=1\\
    W^{\mathbb{N}}\setminus  \mathcal{I}_{j}(\xi_{i,l},\theta), & \text{if } a=0.
  \end{cases}
\]
Note that each $\Xi_{j,i,l}(a)$ is a Borel set and an element of $\mathfrak{F}_{j}$. Here $\mathfrak{F}_{j}$ is the $\sigma$-algebra induced by the random variables $\lambda_1,\ldots,\lambda_{j}$, i.e. $\mathfrak{F}_{j} =
\sigma(\lambda_1,\ldots,\lambda_{j}).$ Moreover, given $n,i\in \N$, $m\geq n,$  $(a_1,\ldots,a_{m})\in \{0,1\}^{m}$ and \(l\in\{0,\ldots, X(i)-1\}\), we define (similarly as in \cref{eq:calF}),
\[
  \mathcal{F}'_{n}(i,l;a_1,\ldots,a_{m}):= \bigcap_{k=1}^{(n-1)} \Xi_{k,i,l}(a_i)
  \;\cap\; \bigcap_{k=n}^m (\Xi_{k,i,l}(a_i)\cap \mathcal{A}_{\eps}(k)).
\]

The proof of \cref{Lemma:prob2} below is similar to the proof of
\cref{lem:value}. We omit a reasonable amount of analogous details so as to avoid repetition.  

\begin{lemma}\label{Lemma:prob2}
  Let $\eps\in(0,1)$, $n\in \N$, $i$ be sufficiently large that $E_i>n$ and $n< m\leq E_i$.
  Denote $\Delta=\lambda_{\max}-\lambda_{\min}$. Then for all \(l\in\{0,\ldots,X(i)-1\}\) and $(a_1, \ldots, a_{m})\in \{0,1\}^{m}$ we have
  \begin{align}\label{BC3}
    \mathbb{P}_{E_i}(\Xi_{m,i,l}(a_m)\mid
    {\mathcal{F}_{n}'}(i,l,a_1,\ldots,a_{m-1}))  
    \leq \begin{cases}
      \left(1+\frac{3\lambda_g^{\eps^2 m}}{\Delta} \right)\cdot \left(1-\frac{2\theta}{\pi}\right) ,&
      \text{if } a_m=1\\
      \left(1+\frac{3\lambda_g^{\eps^2 m}}{\Delta} \right)\cdot \frac{2\theta}{\pi}, & \text{if }
      a_m=0.
    \end{cases}
  \end{align}

\end{lemma}

\begin{proof}
  We just consider the case $a_m=0$. The case where $a_{m}=1$ is analogous. In this case, from the definition of $\Xi_{m,i,l}(0)$ we are interested in those $(\lambda_{1}, \lambda_{2}, \ldots)$ for which  $\pi\prod_{k=1}^m \lambda_k \xi_{i,l}
  \notin V_{\theta}+ \mathbb Z\pi$.  

  Now, let $(\lambda'_1,\ldots,\lambda'_{m-1},\ldots)\in\mathcal{F}_{n}'(i,l;a_1,\ldots,a_{m-1})$.   Keeping in mind that $a_m=0$, we define 
  \begin{align*}
    S'(\lambda_1', \ldots, \lambda_{m-1}'):=&\left\{\lambda_m\in W\colon \pi\prod_{k=1}^m \lambda_k \xi_{i,l}\notin
    V_{\theta}+\mathbb{Z}\pi \text{ and }
  (\lambda_1,\ldots,\lambda_{m-1})=(\lambda'_1,\ldots,\lambda'_{m-1}) \right\}. 
  \end{align*}
  Hence
  \begin{align*}
     S'(\lambda_1', \ldots, \lambda_{m-1}')
    =W\cap\bigcup_{k\in\mathbb{Z}}
    \left(\bigg[\frac{-\theta}{\pi\xi_{i,l}\prod_{j=1}^{m-1}\lambda_j'},\frac{\theta}{\pi\xi_{i,l}\prod_{j=1}^{m-1}\lambda_j'}\bigg)
    +\frac{k}{\xi_{i,l}\prod_{j=1}^{m-1}\lambda_j'}\right).
  \end{align*}
 
Just as in the proof of \cref{lem:value},   the set \(S'(\lambda_1', \ldots, \lambda_{m-1}')\) is a
  union of some large number $P'$ of intervals of some small length $J'= 
  \theta (\pi\xi\prod_{j=1}^{m-1}\lambda_j')^{-1}$, and at most $2$
  intervals of length less than $J'$. By a similar method as used in the proof of \cref{lem:value}, we can prove the following bound
  \begin{equation}\label{eq:J'}
    J'\leq \frac{\theta}{\pi  \lambda_g^{-\eps^2 m}},
  \end{equation}
  and
  \begin{equation}\label{eq:P'}
     P\leq
    2\floor{\Delta\xi_{i,l}\prod_{j=1}^{m-1}\lambda_j'}+4.
  \end{equation} 
  Thus, from a computation much like that appearing in the proof of \cref{lem:value} but
  using the inequalities \cref{eq:J',eq:P'}, we have
  \begin{align*}
    \frac{\mathcal{L}(S'(\lambda_1',\ldots,\lambda_{m-1}'))}{\mathcal L(W)} \leq \frac{J'\times
    (P'+2)}{\Delta}
    &\leq\left(1+\frac{3\lambda_g^{\eps^2 m}}{\Delta} \right) \cdot \frac{2\theta }{\pi}.
  \end{align*}
  This estimate holds uniformly for all \((\lambda_1',\ldots,\lambda_{m-1}',\ldots)\in
  \mathcal{F}_{n}'(i,l;a_1,\ldots,a_{m-1})\). Now, recall that the probability measure $\mathbb P$ is the product of normalised Lebesgue measures on $W^\N$. Hence, 
  \begin{align*}
    \mathbb{P}(\Xi_{m,i,l}(a_m)\mid\mathcal{F}'_{n}(i,l;a_1,\ldots,a_{m-1}))
    \leq \left(1+\frac{3\lambda_g^{\eps^2 m}}{\Delta} \right) \cdot \frac{2\theta }{\pi}. 
  \end{align*}
\end{proof}

Given the probability estimates established above, and the reduction provided by \cref{PD*}, we are in a
position to apply a Borel-Cantelli argument over the frequencies in the sequence $(\xi_{i,l})$. That
is how, in the following lemma, we prove that the polynomial decay inequality is satisfied in the
tail of the sequence $(\xi_{i,l})$, for almost every \(\omega \in B_{\eps}(n)\).

\begin{lemma} \label{Lemma PD1}
  Let $\eps\in(0,1).$ If there exists \(\theta\in(0,\frac{\pi}{4})\), \(\alpha\in(0,1-\frac{2\theta}{\pi})\) and  \(X:\mathbb{N}\to\mathbb{N}\) such
  that if we let $p=1-\frac{2\theta}{\pi}$ and
  \[\sum_{i=1}^{\infty}X(i)\cdot \left(\frac{2\theta}{\pi}\right)^{-\eps i} \cdot e^{ -i
  \left(\alpha \log\left( \frac{\alpha}{p} \right) + (1 - \alpha) \log\left( \frac{1 - \alpha}{1 -
p} \right)\right)}<\infty,\]
  then there exists $\rho>0$ depending upon $\theta,\alpha$ and $\lambda_{g}$, such that for almost all $\omega$ there exists some
  $K_{\omega}>0$ depending upon \(\omega\) such that for all $i>K_{\omega}$ and
   \(l\in\{0,\ldots,X(i)-1\}\) we have
  \begin{align}
    |\widehat{\mu}_\omega(\xi_{i,l})|\leq \xi_{i,l}^{-\rho}.
    \nonumber
  \end{align}
\end{lemma}
\begin{proof}
  Fix $\eps\in(0,1).$ Let $\theta, \alpha$ and $X$ be such that the hypothesis of our lemma is satisfied. By \cref{PD*}, the condition \cref{eq:suff} is a
  sufficient condition to deduce our desired conclusion. Hence to prove our lemma, it suffices to prove that for our specific $\theta$ and $\alpha$ we have
  \begin{equation}
  	\label{eq:uniform count}
  	\liminf_{i\to \infty}\min_{\xi_{i,l}\in S_i}\frac{\#\{j\leq E_i\colon \pi\prod_{k=1}^j \lambda_k
  		\xi_{i,l} \in V_{\theta} + \mathbb Z\pi \}}{i}\geq \alpha
  \end{equation}
  for almost all $\omega$. Moreover, because $\cup_{n\in \N} B_{\epsilon}(n)$ equals $W^{\N}$ modulo a set of measure zero, to prove our result it suffices to show that \cref{eq:uniform count} holds for almost every $\omega\in B_{\epsilon}(n)$ for any $n\in \N$. With this in mind we now fix $n\in \N$. Let $i\in \N$ be sufficiently large that $E_i>n$. Then for a given \(l\in\{0, \ldots, X(i)-1\}\) we have
  \begin{align*}
    &\mathbb{P}\left(\left\{\omega\colon \frac{\#\{j\leq E_i\colon
	\pi\prod_{k=1}^j \lambda_k \xi_{i,l} \in V_{\theta}+ \mathbb Z\pi   \}}{i} <
    \alpha\right\} \cap B_{\eps}(n)\right)\nonumber\\ 
    \leq &\sum_{\substack{(a_1,\ldots,a_{E_i})\in \{0,1\}^{E_i}\\\sum_{k=1}^{E_i}a_i<\alpha i}}
    \mathbb{P}_{E_i}\left[\bigcap_{k=1}^{n-1}\Xi_{k,i,l}(a_i)\cap \left(\bigcap_{j=n}^{E_i}
    \Xi_{j,i,l}(a_j)\cap A_\eps(j)\right)\right]. 
  \end{align*}

  By the tower law of conditional
  probabilities and using the same incremental measurability argument as used in the proof of \cref{lem:msr}, for $(a_1,\ldots,a_{E_i})\in \{0,1\}^{E_i}$ we have
  \begin{align}\label{BC2}
  &\mathbb{P}_{E_i}\left[\bigcap_{k=1}^{n-1}\Xi_{k,i,l}(a_i)\cap \left(\bigcap_{j=n}^{E_i}
  \Xi_{j,i,l}(a_j)\cap A_\eps(j)\right)\right]\nonumber\\
    =& \mathbb{P}_{E_i}(\mathcal{F}'_{n}(i,l;a_1,\ldots,a_{E_i}))\nonumber\\
    \le& \prod_{k=n+1}^{E_i} \mathbb{P}_{E_i}(\Xi_{k,i,l}(a_k)\cap A_{\eps}(k)\mid
    \mathcal{F}'_{n}(i,l;a_1,\ldots,a_{k-1}))\nonumber\\
    \le& \prod_{k=n+1}^{E_i} \mathbb{P}_{E_i}(\Xi_{k,i,l}(a_k)\mid
    \mathcal{F}'_{n}(i,l;a_1,\ldots,a_{k-1})).
  \end{align}
 Suppose now that $(a_1,\ldots,a_{E_i})\in \{0,1\}^{E_i}$ is such that $\sum_{k=1}^{E_i}a_i=\ell$, then by \cref{Lemma:prob2} and using the inequality $\frac{2\theta}{\pi}<1-\frac{2\theta}{\pi}\footnote{This follows from our assumption $\theta\in (0,\pi/4)$.},$ it can be shown that 
\begin{equation}
	\label{eq:Prob bound}
	\prod_{k=n+1}^{E_i} \mathbb{P}_{E_i}(\Xi_{k,i,l}(a_k)\mid
	\mathcal{F}'_{n}(i,l;a_1,\ldots,a_{k-1}))\leq  \prod_{k=n+1}^{E_i}
	\left(1+\frac{3\lambda_g^{\eps^2 k}}{\Delta} \right)\cdot
	\left(1-\frac{2\theta}{\pi}\right)^{\ell}\cdot \left(\frac{2\theta}{\pi}\right)^{E_i-n-\ell}.
\end{equation}

 Thus, combining \cref{BC2,BC3,eq:Prob bound}, and considering all sequences
  \((a_1,\ldots,a_{E_i})\in \{0,1\}^{E_i}\) such that \(\sum_{k=1}^{E_i}a_i<\alpha i\), we have the bound
  \begin{align*}
    &\mathbb{P}\left(\left\{\omega\colon \frac{\#\{j\leq E_i\colon
    	\pi\prod_{k=1}^j \lambda_k \xi_{i,l} \in V_{\theta}+ \mathbb Z\pi   \}}{i} <
    \alpha\right\} \cap B_{\eps}(n)\right)\\
  \leq&\sum_{\ell=0}^{\floor{\alpha i}}\binom{E_i}{\ell}\cdot\prod_{k=n+1}^{E_i}
    \left(1+\frac{3\lambda_g^{\eps^2 k}}{\Delta} \right)\cdot
    \left(1-\frac{2\theta}{\pi}\right)^{\ell}\cdot \left(\frac{2\theta}{\pi}\right)^{E_i-n-\ell}.
  \end{align*}
  Recall that \(E_i=\floor{(1-\eps)i}\leq i\) and therefore
\(\binom{E_i}{l}\leq \binom{i}{l}\).
  Using this bound we have
  \begin{align*}
   &\sum_{\ell=0}^{\floor{\alpha i}}\binom{E_i}{\ell}\cdot\prod_{k=n+1}^{E_i}
   \left(1+\frac{3\lambda_g^{\eps^2 k}}{\Delta} \right)\cdot
   \left(1-\frac{2\theta}{\pi}\right)^{\ell}\cdot \left(\frac{2\theta}{\pi}\right)^{E_i-n-\ell}\\
  \leq&\sum_{\ell=0}^{\floor{\alpha i}}\binom{i}{\ell}\cdot\prod_{k=n+1}^{E_i}
    \left(1+\frac{3\lambda_g^{\eps^2 k}}{\Delta} \right)\cdot
    \left(1-\frac{2\theta}{\pi}\right)^{\ell}\cdot \left(\frac{2\theta}{\pi}\right)^{E_i-n-\ell}\\
    =& \left(\frac{2\theta}{\pi}\right)^{E_i-n-i} \cdot\prod_{k=n+1}^{E_i}
    \left(1+\frac{3\lambda_g^{\eps^2 k}}{\Delta} \right)\cdot \sum_{\ell=0}^{\floor{\alpha
  i}}\binom{i}{\ell} \left(1-\frac{2\theta}{\pi}\right)^{\ell}\left(\frac{2\theta}{\pi}\right)^{i-\ell}.
  \end{align*}
  We have $$\prod_{k=1}^{E_i}\left(1+\frac{3\lambda_g^{\eps^2 k}}{\Delta} \right)\le
  \prod_{k=1}^{\infty}\left(1+\frac{3\lambda_g^{\eps^2 k}}{\Delta} \right)=:C<\infty.$$ Using this bound and collecting the above estimates, we have shown that 
  \begin{align}
  	\label{eq:before Chernoff}
    &\mathbb{P}\left(\left\{\omega\colon \frac{\#\{j\leq E_i\colon
    	\pi\prod_{k=1}^j \lambda_k \xi_{i,l} \in V_{\theta}+ \mathbb Z\pi   \}}{i} <
    \alpha\right\} \cap B_{\eps}(n)\right)\nonumber \\
    \le& C \left(\frac{2\theta}{\pi}\right)^{E_i-n-i} \cdot
  \sum_{\ell=0}^{\floor{\alpha i}}\binom{i}{\ell}
    \left(1-\frac{2\theta}{\pi}\right)^{\ell}\left(\frac{2\theta}{\pi}\right)^{i-\ell}.
  \end{align}

Observe now that $$\sum_{\ell=0}^{\floor{\alpha i}}\binom{i}{\ell}
  \left((1-\frac{2\theta}{\pi}\right)^{\ell}\left(\frac{2\theta}{\pi}\right)^{i-\ell}$$ is the binomial
  sum of the first \(\floor{\alpha i}+1\) terms in
  \(\left((1-\frac{2\theta}{\pi})+\frac{2\theta}{\pi}\right)^i\).
  Furthermore, since $\alpha<1-\frac{2\theta}{\pi}$, using the Chernoff bound we have
  \begin{equation} 
  	\label{eq:Chernoff}
  \sum_{l=0}^{\floor{\alpha i}}\binom{i}{l}
    \left(1-\frac{2\theta}{\pi}\right)^{l}\left(\frac{2\theta}{\pi}\right)^{i-l}\leq e^{ -i
      \left(\alpha \log\left( \frac{\alpha}{p} \right) + (1 - \alpha) \log\left( \frac{1 - \alpha}{1 -
    p} \right)\right)}.
  \end{equation}
We recall that for notational convenience we let $p=1-\frac{2\theta}{\pi}$. Combining \cref{eq:before Chernoff} with \cref{eq:Chernoff}, we see that for any \(0\leq l \leq X(i)-1\) we have 
  \begin{equation*}
    \begin{aligned}
      &\mathbb{P}\left(\left\{\omega\colon \frac{\#\{j\leq E_i\colon
      	\pi\prod_{k=1}^j \lambda_k \xi_{i,l} \in V_{\theta}+ \mathbb Z\pi   \}}{i} <
      \alpha\right\} \cap B_{\eps}(n)\right)\\
      \leq& C\cdot \left(\frac{2\theta}{\pi}\right)^{E_i-n-i}\cdot   e^{ -i
	\left(\alpha \log\left( \frac{\alpha}{p} \right) + (1 - \alpha) \log\left(
      \frac{1 - \alpha}{1 - p} \right)\right)}.
    \end{aligned}
  \end{equation*}
  This in turn implies that 
  \begin{equation}\label{eq:probbound2}
  	\begin{aligned}
  		&\mathbb{P}\left(\left\{\omega\colon \min_{l\in \{0,\ldots, X(i)-1\}}\frac{\#\{j\leq E_i\colon
  			\pi\prod_{k=1}^j \lambda_k \xi_{i,l} \in V_{\theta}+ \mathbb Z\pi   \}}{i} <
  		\alpha\right\} \cap B_{\eps}(n)\right)\\
  		\leq& C\cdot X(i)\cdot \left(\frac{2\theta}{\pi}\right)^{E_i-n-i} \cdot e^{ -i
  			\left(\alpha \log\left( \frac{\alpha}{p} \right) + (1 - \alpha) \log\left(
  			\frac{1 - \alpha}{1 - p} \right)\right)}.
  	\end{aligned}
  \end{equation}
By \cref{eq:probbound2} and our assumptions on $\theta$, $\alpha$ and $X$, we have\footnote{Notice that because we are summing over $i$ the additional $n$ term appearing in the exponent for $\frac{2\theta}{\pi}$ does not influence the convergence of this sum.} 
$$\sum_{i=1}^{\infty} \mathbb{P}\left(\left\{\omega\colon \min_{l\in \{0,\ldots, X(i)-1\}}\frac{\#\{j\leq E_i\colon
	\pi\prod_{k=1}^j \lambda_k \xi_{i,l} \in V_{\theta}+ \mathbb Z\pi   \}}{i} <
\alpha\right\} \cap B_{\eps}(n)\right)<\infty.$$
Hence, by the Borel-Cantelli Lemma, almost every $\omega\in B_{\epsilon}(n)$ satisfies \cref{eq:uniform count}. As previously remarked, by \cref{PD*} this implies that the desired conclusion holds for almost every $\omega\in B_{\epsilon}(n).$ Since $n$ was arbitrary and $\cup_{n}B_{\epsilon}(n)$  equals $W^{\N}$ up to a set of measure zero, it follows that the desired conclusion holds for almost every $\omega$.  
\end{proof}

The following lemma shows that we can construct a suitable function $X:\N\to\N$. $X$ needs to satisfy two properties. It needs to grow sufficiently slowly that \cref{Lemma PD1} holds, and it needs to grow sufficiently quickly so that having the desired polynomial decay rate along the sequence $(\xi_{i,l})_{i,l}$ is sufficient for establishing the decay rate over the whole of $\R$. With the second property in mind, we need to introduce some notation to formalise the Lipschitz continuity of the $\hat{\mu}_{\omega}$ functions. The Fourier transforms $\widehat\mu_{\omega}$
are Lipschitz continuous and the Lipschitz constant can be chosen uniformly to apply to all in $\omega$, since for all $\omega$, the support
satisfies
$\Lambda_\omega\subset [0, (1-\lambda_{\max})^{-1}]$ (see e.g. \cite[Equation
(3.19)]{mattila_2015}). Let us denote this uniform constant by $H$. So we have 
$$|\hat{\mu}_{\omega}(\xi)-\hat{\mu}_{\omega}(\xi')|\leq H|\xi-\xi'|$$ for all $\xi,\xi'\in\mathbb{R}$ and $\omega\in W^{\N}$.

\begin{lemma}
	\label{Lemma:choose X}
	Let $\eps\in(0,1/2)$. Then there exists $\theta\in(0,\pi/4),$ $\alpha\in (0,1-\frac{2\theta}{\pi})$ and $X:\N\to \N$  such that for all large $i$ large enough we have
	\begin{itemize}
		\item[(P1)] \(H\frac{\lambda_g^{-i-1}}{X(i)} < \lambda_g^{i}\),
		\item[(P2)]  \(X(i)\cdot \left(\frac{2\theta}{\pi}\right)^{-\eps i} \cdot e^{ -i \left(\alpha
			\log\left( \frac{\alpha}{p} \right) + (1 - \alpha) \log\left( \frac{1 - \alpha}{1 - p}
			\right)\right)}<e^{- i}\).
	\end{itemize}
\end{lemma}
\begin{proof}
Rearranging $(P1)$ and $(P2),$ we see that our result follows if we can find $\theta,\alpha$ and $X$ such that
\begin{equation}
	\label{eq:reduction}
	H\lambda_g^{-2i-1}<X(i)<\left(\frac{2\theta}{\pi}\right)^{\eps i} \cdot e^{ i
		\left(\alpha \log\left( \frac{\alpha}{p} \right) + (1 - \alpha) \log\left( \frac{1 - \alpha}{1 -
			p} \right)\right)}\cdot e^{- i} 
\end{equation}for all $i$ sufficiently large. Suppose now that we can find $\theta$ and $\alpha$ such that 
\begin{equation}
	\label{eq:Find X}
	 H\lambda_g^{-2i-1}+1<\left(\frac{2\theta}{\pi}\right)^{\eps i} \cdot e^{ i
		\left(\alpha \log\left( \frac{\alpha}{p} \right) + (1 - \alpha) \log\left( \frac{1 - \alpha}{1 -
			p} \right)\right)}\cdot e^{-i} 
\end{equation} for all $i$ sufficiently large. Then we can find a function $X$ so that \cref{eq:reduction} holds for all $i$ sufficiently large. We simply choose any integer between the left hand side of \cref{eq:reduction} and the right hand side of \cref{eq:reduction} then let $X(i)$ equal this integer. Thus to complete our proof we need to show that \cref{eq:Find X} holds for all $i$ sufficiently large. Taking logarithms, dividing by $i$, then taking the limit as $i\to\infty,$ we see that \cref{eq:Find X} holds for $i$ sufficiently large if $\theta$ and $\alpha$ are such that
\begin{equation} \label{eq:pdB}
-2\log\lambda_g< \alpha \log\left( \frac{\alpha}{p} \right) + (1 - \alpha) \log\left(
\frac{1 - \alpha}{1 - p} \right)+\eps\log (1-p)-1.
\end{equation} We can rewrite \cref{eq:pdB} in the equivalent form
\begin{equation} \label{eq:pdC}
	-2\log\lambda_g< \alpha \log\left( \frac{\alpha}{p} \right) + (1 - \alpha) \log\left(
	1 - \alpha \right)+(\eps-1+\alpha)\log (1-p)-1.
\end{equation}
We now set $\alpha=p/2$ (recall that $p=1-\frac{2\theta}{\pi})$ so \cref{eq:pdC} becomes
\begin{equation} \label{eq:pdD}
	-2\log\lambda_g< \frac{p}{2} \log\left( \frac{1}{2} \right) + \left(1 - \frac{p}{2}\right) \log\left(
	1 - \frac{p}{2} \right)+\left(\eps-1+\frac{p}{2}\right)\log (1-p)-1.
\end{equation}
Then as $\theta\to 0$ and therefore $p\to 1$, the right hand side of \cref{eq:pdC} tends to infinity. This is because for $\epsilon\in (0,1/2)$ we have $$\lim_{p\to 1}\left(\eps-1+\frac{p}{2}\right)\log (1-p)=\infty,$$ and the other terms on the right hand side of \cref{eq:pdD} remain bounded as $p\to 1$. Since the left hand side of \cref{eq:pdD} does not depend upon $\theta$, we can therefore choose $\theta$ such that \cref{eq:pdD} holds. For this value of $\theta,$ if we take $\alpha=p/2$, then it follows from the above that \cref{eq:pdC} holds. This completes our proof. 
\end{proof}

Equipped with \cref{Lemma PD1} and \cref{Lemma:choose X} we can now complete the proof of \cref{main3}.

\begin{proof}[Proof of \cref{main3}]
	Let $\epsilon\in (0,1/2)$. Then by \cref{Lemma:choose X} there exists $\theta\in(0,\pi/4),$ $\alpha\in (0,1-\frac{2\theta}{\pi})$ and $X:\N\to \N$ such that $(P1)$ and $(P2)$ of this lemma are satisfied for $i$ sufficiently large. For these choices of $\epsilon, \theta, \alpha$ and $X$, if follows from (P2) that 
	 $$\sum_{i=1}^{\infty}X(i)\cdot \left(\frac{2\theta}{\pi}\right)^{-\eps i} \cdot e^{ -i
		\left(\alpha \log\left( \frac{\alpha}{p} \right) + (1 - \alpha) \log\left( \frac{1 - \alpha}{1 -
			p} \right)\right)}<\infty.$$ Then by \cref{Lemma PD1}, there exists $\rho>0$ depending upon $\theta$, $\alpha$ and $\lambda_{g},$ such that for almost all $\omega$ there exists $K_{\omega}>0$ depending upon $\omega$ such that for all $i>K_{\omega}$ and $l\in \{0,\ldots, X(i)-1\}$ we have 
		\begin{equation}
			\label{eq:subsequence decay}		|\hat{\mu}_{\omega}(\xi_{i,l})|\leq \xi_{i,l}^{-\rho}.
			\end{equation}
			Let us now fix such a $\rho.$ Let us now also make an arbitrary choice of $\omega$ belonging to the full measure set for which \ref{eq:subsequence decay} holds for all $i>K_{\omega}$ and $l\in \{0,\ldots, X(i)-1\}$. Let $\xi \in I_{i}$ for some $i>K_{\omega}$ be arbitrary. Then there exists $l\in \{0,\ldots,X(i)-1\}$ such that 
			\begin{equation}
				\label{eq:close frequency}
				|\xi-\xi_{i,l}|\leq \frac{\lambda_{g}^{-i-1}-\lambda_{g}^{-i}}{X(i)}.
			\end{equation}
			 For this choice of $l$, it follows from \cref{eq:subsequence decay}, \cref{eq:close frequency}, and the Lipschitz property of $\hat{\mu}_{\omega}$ that
			\begin{equation}
				\label{eq:Part way to decay}
				|\hat{\mu}_{\omega}(\xi)|\leq |\hat{\mu}_{\omega}(\xi_{i,l})|+H \frac{\lambda_{g}^{-i-1}-\lambda_{g}^{-i}}{X(i)}\leq \xi_{i,l}^{-\rho}+ +H \frac{\lambda_{g}^{-i-1}}{X(i)}.
			\end{equation} Since $\xi,\xi_{i,l}\in I_{i}$ it follows that $\xi\leq \xi_{i,l}\lambda_{g}^{-1}$ and $\xi\leq \lambda_{g}^{-i-1}$. Using these inequalities together with $(P1)$, it follows from \cref{eq:Part way to decay} that 
			$$|\hat{\mu}_{\omega}(\xi)|\leq \frac{\xi^{-\rho}}{\lambda_{g}^{\rho}}+\lambda_{g}^{i}\leq \frac{\xi^{-\rho}}{\lambda_{g}^{\rho}}+\frac{\xi^{-1}}{\lambda_{g}}.$$ Since $\xi\in I_{i}$ for some $i>K_{\omega}$ was arbitrary, if follows that for  $\rho'=\min \{\rho,1\}$ and $C>0$ sufficiently large, we have $|\hat{\mu}_{\omega}(\xi)|\leq C\xi^{-\rho'}$ for all $\xi>0$. Our proof is almost complete, it remains to consider those $\xi<0$. Using \cref{lem:form} it follows that $|\hat{\mu}_{\omega}(\xi)|=|\hat{\mu}_{\omega}(-\xi)|$ for any $\xi \in \R$. Thus for this choice of $\rho'$ and $C$ we in fact have $|\hat{\mu}_{\omega}(\xi)|\leq C|\xi|^{-\rho'}$ for all $\xi\neq 0.$ Since $\omega$ was an arbitrary choice from a full measure set our result follows.
\end{proof}

\bibliographystyle{alpha}
\bibliography{ref}

\end{document}